\documentclass[a4paper,12pt]{article}
\usepackage{amsmath}
\usepackage{amsthm}
\usepackage{amssymb}
\usepackage{float}
\usepackage[english]{babel}
\usepackage{graphicx}
\newtheorem{thmenv}{Theorem}
\newtheorem{lemenv}[thmenv]{Lemma}

\newtheorem{corenv}[thmenv]{Corollary}
\theoremstyle{definition}

\usepackage{url}
\begin{document}
\title{Counting Connected Set Partitions of Graphs}
\author{
Frank Simon\footnote{E-mail: {\tt simon@hs-mittweida.de}, Hochschule Mittweida (FH)  University of Applied Sciences, Fakult\"at Mathematik/Naturwissenschaften/Informatik, 
Technikumplatz 17, D-09648 Mittweida, Germany}\, 
\and 
Peter Tittmann\footnote{E-mail: {\tt peter@hs-mittweida.de}, Hochschule Mittweida (FH)  University of Applied Sciences, Fakult\"at Mathematik/Naturwissenschaften/Informatik, 
Technikumplatz 17, D-09648 Mittweida, Germany}\,
\and 
Martin Trinks\footnote{E-mail: {\tt trinks@hs-mittweida.de}, Hochschule Mittweida (FH)  University of Applied Sciences, Fakult\"at Mathematik/Naturwissenschaften/Informatik, 
Technikumplatz 17, D-09648 Mittweida, Germany}\,
}
\maketitle 
\begin{minipage}[c]{6cm}
The authors Frank Simon and Martin Trinks receive a grant from the European Union.
\end{minipage}
\hfill
\begin{minipage}[c]{6cm}
\includegraphics[width=3cm]{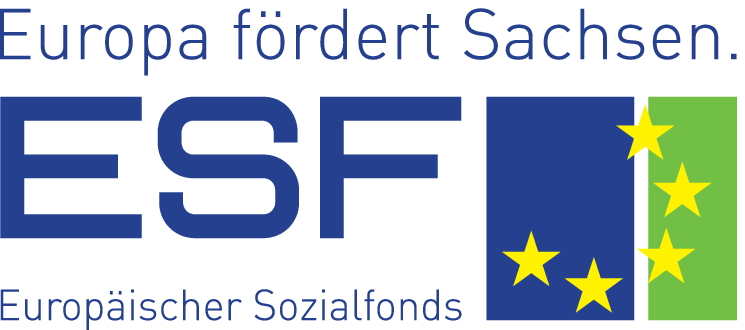}
\includegraphics[width=2cm]{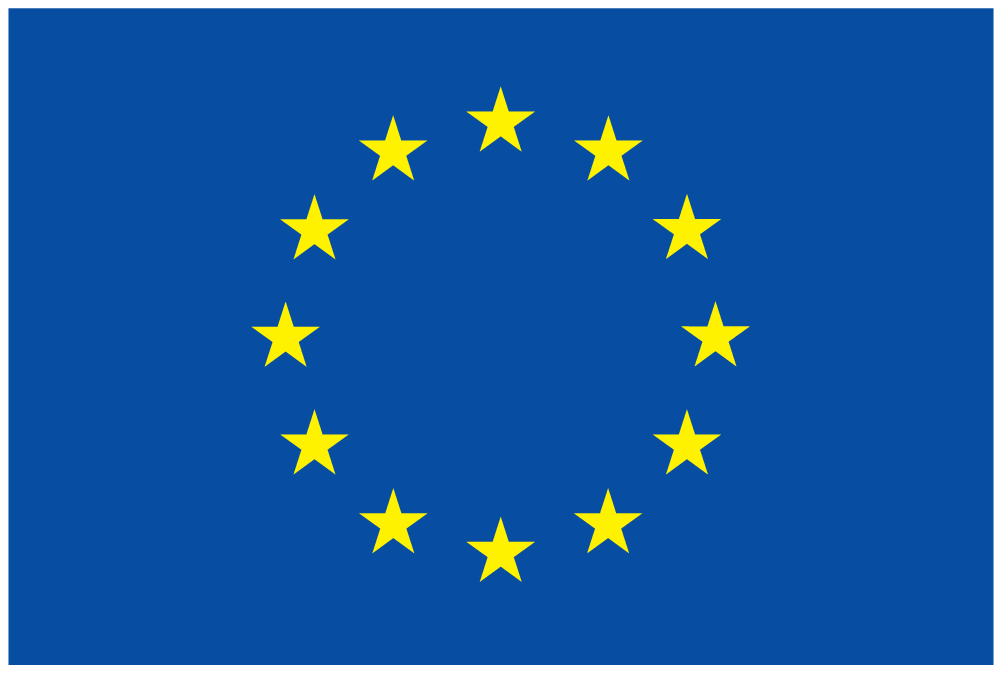}
\end{minipage}
\begin{abstract}
\noindent
Let $G=(V,E)$ be a simple undirected graph with $n$ vertices then a set partition $\pi=\{V_1, \ldots, V_k\}$ of the vertex set of $G$ is a connected set partition if each subgraph $G[V_j]$ induced by the blocks $V_j$ of $\pi$ is connected for $1\le j\le k$. Define $q_{i}(G)$ as the number of connected set partitions in $G$ with $i$ blocks. The partition polynomial is then $Q(G, x)=\sum_{i=0}^n q_{i}(G)x^i$. This paper presents a splitting approach to the partition polynomial on a separating vertex set $X$ in $G$ and summarizes some properties of the bond lattice. Furthermore the bivariate partition polynomial $Q(G,x,y)=\sum_{i=1}^n \sum_{j=1}^m q_{ij}(G)x^iy^j$ is briefly discussed, where $q_{ij}(G)$ counts the number of connected set partitions with $i$ blocks and $j$ intra block edges. Finally the complexity for the bivariate partition polynomial is proven to be $\sharp P$-hard.
\end{abstract}
\clearpage
\section{Introduction}

One of the best-studied graph polynomials is the chromatic polynomial $P(G,x)
$, giving the number of proper vertex colorings of an undirected graph $G=(V,E)$
with at most $x$ colors (see e.g. \cite{BirLew46, read68, Stanley73, Tutte70,
Tutte74}). Rota \cite{rota:OnTheFoundationsOfCombinatorialTheory} showed that the chromatic polynomial of a graph
$G$ is uniquely defined by its bond lattice $\Pi_c(G)$. The bond lattice can be
defined as a sublattice of the partition lattice $\Pi\left(  V\right)  $ on
$V$. A set partition $\pi\in\Pi\left(  V\right)  $ belongs to $\Pi_c(G)$ iff all
blocks of $\pi$ induce connected subgraphs of $G$.

We investigate here the rank-generating function of $\Pi_c(G)$, which we call the
\emph{partition polynomial} $Q(G,x)$. There are two ways to consider $Q\left(
G,x\right)  $, namely from an order-theoretic point of view (as
rank-generating function) or as a graph polynomial. We pursue here the second
way. The first natural question in this context is: Does the partition
polynomial $Q(G,x)$ determine the chromatic polynomial $P(G,x)$? We will show
that this is not the case. Even the converse statement is false. There are
pairs of non-isomorphic graphs with coinciding chromatic polynomials but
different partition polynomials. The next interesting problem is the
derivation of graph properties and graph invariants from the partition
polynomial. Which graphs are uniquely determined by their partition
polynomials? We call non-isomorphic graphs with coinciding partition
polynomial \emph{partition-equivalent}. Can we characterize \emph{partition}-equivalent graphs?

The computation of the partition polynomial is the next challenge. The obvious
way, to list all connected vertex partitions (partitions of $\Pi_c(G)$) and count
them with respect to the number of blocks is even for small graphs often too
laborious. Consequently, each method that simplifies the computation of the
partition polynomial is welcome.  

Thus, the next task is the identification of graph classes
permitting a polynomial-time computation of the partition polynomial. Let
$G=\left(  V,E\right)  $ be a graph and $G_{1}$ and $G_{2}$ edge-disjoint
subgraphs of $G$ that have a vertex set $U=V\left(  G_{1}\right)  \cap
V\left(  G_{2}\right)  $ in common. Then a \emph{splitting formula} permits to
find $Q\left(  G,x\right)  $ by separate computation of certain polynomials
assigned to $G_{1}$ and $G_{2}$. The here presented splitting formula for
the partition polynomial is the first step to find a polynomial time algorithm for
graphs of bounded treewidth.

Edges linking different blocks of a connected set partition form a \emph{cut set}.
An edge cut of $G$ corresponds to a cut set defined by a connected two-block
partition. In order to keep track of the number of edges forming a cut set, we
extend the partition polynomial into a bivariate polynomial. 

The paper is organized as follows. Section 2 provides a short introduction
into set partitions and their order properties. Connected set partitions of the
vertex set of a given graph are introduced in Section 3. The main object of
this paper, the partition polynomial $Q\left(  G,x\right)  $, is defined in
Section 4. This section provides also basic properties of the partition
polynomial, some graph invariants that can be derived from $Q\left(
G,x\right)  $, recurrence formulae, and results
for special graphs. Section 5 deals with one of the main results of the paper
- the splitting formula for the partition polynomial. The last section presents
the two-variable extension of the partition polynomial and some examples of
non-isomorphic partition-equivalent graphs as well as pairs of graphs that are
chromatically equivalent and not partition-equivalent and vice versa.

Finally, we can show that the computation of the extended partition polynomial $Q(G,x,y)$ is \#P-hard, whereas complexity results for the simple partition polynomial $Q(G,x)$ are still not known.

\section{Set partitions} 
A set partition $\pi=\{X_1, \ldots, X_k\}$ of a finite set $X$ is a set of mutually disjoint and nonempty subsets $X_i$, the blocks of $X$, so that the union of the $X_i$ is $X$. The number of blocks of a set partition $\pi$ is denoted by $|\pi|$ and the set of all set partitions of $X$ by $\Pi(X)$. The number of set partitions of an $n$-element set with exactly $k$ blocks is called the Stirling number of the second kind, which is denoted by $S(n,k)$. The Bell number $B(n)$ is the number of all set partitions of an $n$-element set, hence $B(n)=\sum_{k=0}^n S(n,k)$. Note that $B(0)$ and $S(0,0)$ equals $1$, as there is exactly one set partition of the empty set with no blocks, namely $\emptyset$.

Let $\sigma,\pi\in\Pi(X)$ and set $\sigma\le\pi$ if every block of $\sigma$ is a subset of a block in $\pi$, then $(\Pi(X), \le)$ becomes a poset. 
The maximal element $\hat{1}$ of this poset is the set partition that has only one block and the minimal element $\hat{0}$ is the set partition that consists only of singleton blocks. This poset is even more a lattice, i.e. for every two partitions $\pi,\sigma\in\Pi(X)$ there exists a smallest upper bound $\pi\vee\sigma$ and a greatest lower bound $\pi\wedge \sigma$ in $\Pi(X)$. 

Assume that $U$ is a subset of $X$, then $\pi=\{X_1,\ldots, X_k\}\in\Pi(X)$ induces the set partition $\sigma=\{U_1,\ldots,U_l\}\in\Pi(U)$ in $U$ by setting $U_i=X_i\cap U$, so that only the nonempty blocks $U_i$ are taken over to $\sigma$. The notation $\sigma=\pi\sqcap U$ is used to indicate that $\pi$ induces $\sigma$ in $U$. 

Let $X,Y$ be finite sets and $\pi=\{X_1, \ldots, X_k\}\in\Pi(X)$, $\sigma=\{Y_1,\ldots, Y_l\}\in \Pi(Y)$ set partitions of $X$ and $Y$, respectively. Then $\pi\sqcup \sigma\in \Pi(X\cup Y)$ denotes the smallest upper bound $\pi'\vee\sigma'$ of the set partitions $\pi'\in\Pi(X\cup Y)$ and $\sigma'\in\Pi(X\cup Y)$, where $\pi'$ consists of the blocks $X_i$ and the remaining blocks being singletons and $\sigma'$ consists of the blocks $Y_i$ and the remaining blocks being singletons.

\section{Connected set partitions}
A simple undirected graph or short graph is a pair $G=(V,E)$ consisting of a finite set $V$, the vertices, and a subset $E\subseteq V^{(2)}$, the edges, of the two element subsets of $V$. If $X\subseteq V$ is a vertex subset of $G$, then $G[X]$ denotes the subgraph induced by $X$ that has the vertex set $X$ and all edges $\{u,v\}\in E$ of $G$ that have both end vertices $u$ and $v$ in $X$.

Let $\pi=\{V_1, \ldots, V_k\}\in \Pi(V)$ then $\pi$ is a connected set partition in $G$ if for all $V_i$ the subgraphs $G[V_i]$ induced by $V_i$ are connected. The set of all connected set partitions in a graph $G$ is denoted by $\Pi_c(G)$. 

Consider the graph depicted in Figure~\ref{fig:exampleGraph}, that has 89 distinct connected set partitions,
\begin{figure}[h]
\begin{center}
\includegraphics[scale=0.6]{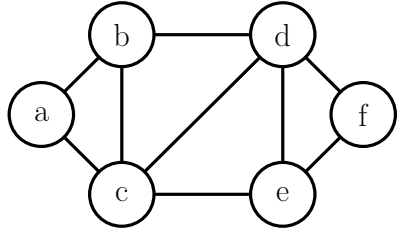}
\caption{Example graph\label{fig:exampleGraph}}
\end{center}
\end{figure}
two of them are $\pi_1=\{\{a,b\},\{c\},\{d,e\},\{f\}\}$ and $\pi_2=\{\{a,b,c\},\{d,e,f\}\}$. For the complete graph $K_{n}=(V,V^{(2)})$ with $n$ vertices it is 
$\Pi_{c}\left(K_{n}\right) =\Pi(V)$.

There is a constructive way to generate all connected set partitions of a given graph $G=(V,E)$. In preparation of this construction assume that the vertices $v\in V$ of the graph $G$ are considered as the one element sets $\{v\}$. Let $e\in E$, then $G/e$ is the graph obtained from $G$ by contraction of the edge $e$. If an edge $e=\{R,S\}$ is contracted the vertex obtained by merging the two vertices $R$ and $S$ is denoted by $R\cup S$ and is the union of the vertex subsets $R$ and $S$ of the original graph. As a result of any sequence of edge contractions of $G$, the graph $H=(W,F)$ is obtained, where $W$ forms a connected set partition of $V$. The set of all connected set partitions that can be obtained by such edge contractions of $G$ is again $\Pi_c(G)$. Hence $\Pi_c(G)$ is sometimes called the lattice of contractions.

Another interpretation of the set of connected set partitions of a graph $G=(V,E)$ is given by the following procedure. Denote by $\{G\}\in\Pi(V)$ the set partition of the vertex set induced by the connected components of the graph $G=(V,E)$, hence $\{G\}=\{V\}$ iff $G$ is connected. Let $H=\left( V,F\right)$ with $F\subseteq E$ be a subgraph of $G$. The components of $H$ then \emph{induce} a set partition $\left\{  H\right\} \le\{G\}$. Two vertices $u,v\in V$ are elements of the same block of $\left\{H\right\}  $ iff these vertices belong to the same component of $H$. Observe that a set partition $\pi\in\Pi\left(V\right)$ belongs to $\Pi_{c}\left(G\right)  $ iff there is an edge subset $F\subseteq E$ and a spanning subgraph $H=(V,F)$ such that $\pi=\left\{H\right\}$ holds. 
\begin{lemenv}
Let $G=\left(  V,E\right)  $ be a forest with $m$ edges. Then $\left\vert
\Pi_{c}\left(  G\right)  \right\vert =2^{m}$.
\end{lemenv}

\begin{proof}
Each edge subset $F\subseteq E$ of the forest $G=\left(  V,E\right)  $
generates a connected set partition $\left\{(V,F)\right\} $. Conversely, if $\pi\in\Pi_{c}\left(  G\right)  $ then there exists a subset $F\subseteq E$
with $\left\{(V,F)\right\}=\pi$. Assume there is a different edge set
$F^{\prime}\subseteq E$ with $\left\{(V,F^{\prime})\right\}=\pi$. In this
case there are components $K=(X,J)$ and $K'=(X, J^{\prime})$ of $H=(V,F)$ and $H'=(V, F')$, respectively with the same vertex set but different edge sets. Consequently, $K''=\left(  X,J\cup J^{\prime}\right)$ contains a cycle which contradicts the premise that $G$ is a forest. Hence each edge subset induces a unique set partition of $\Pi_{c}\left(  G\right)  $.
\end{proof}
Further properties of the set $\Pi_{c}\left(  G\right)  $ are easy to verify:
\begin{enumerate}
\item If $\sigma,\pi\in\Pi_{c}\left(  G\right)  $ then $\sigma\vee
\pi\in\Pi_{c}\left(  G\right)  $.
\item $\Pi_{c}\left(  G\right)  $ is a \emph{geometric lattice}, i.e.
each element of $\Pi_{c}\left(  G\right)  $ is the smallest upper bound of some elements covering $\hat{0}$. These connected set partitions consist of one block with exactly two elements and otherwise only singleton blocks. Hence every two element block of such an atomic connected set partition can be identified with an edge of $G$.
\end{enumerate}%
The lattice $\Pi_c(G)$ of connected set partitions also occurs under the name bond lattice or lattice of contractions in the literature \cite{stanley:ASymmetricFunctionGeneralizationOfTheChromaticPolynomialOfAGraph}.
\begin{figure}
[h]
\begin{center}
\includegraphics[scale=0.6]
{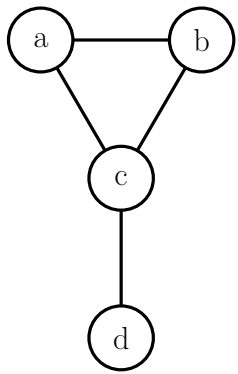}%
\caption{Small example graph}%
\label{graph2}%
\end{center}
\end{figure}
\begin{figure}
[ptb]
\begin{center}
\includegraphics
{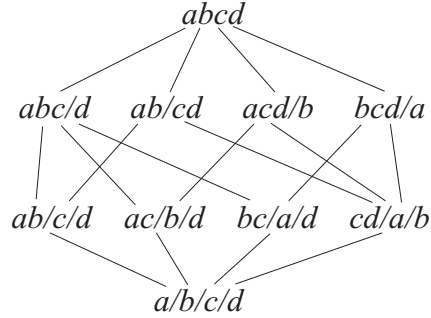}%
\caption{Lattice of connected set partitions}%
\label{lattice1}%
\end{center}
\end{figure}
Figure \ref{graph2} shows a graph with vertex set $\left\{  a,b,c,d\right\}
$. Figure \ref{lattice1} represents the lattice of connected set partitions of
this graph.

A \emph{vertex coloring} of $G$ is a function $\phi:V\rightarrow X$ from the set of vertices of $G$ to a finite set of colors $X$. A vertex
coloring $\phi$ is \emph{proper} if $\left\{  u,v\right\}  \in E$ implies
$\phi\left(  u\right)  \neq\phi\left(  v\right)  $. Denote by $P\left(
G,x\right)  $ the chromatic polynomial of $G$, i.e. the number of proper
vertex colorings of $G$ using a color set with exactly $x$ colors. For
$x\in\mathbb{N}$ let $X$ be a set of $x$ colors. For $\pi\in\Pi\left(
V\right)  $ let $f_{G}\left(  \pi,x\right)  $ be the number of vertex
colorings of $G$ with colors of $X$ such that all vertices that belong to one
block of $\pi$ are colored alike and such that end vertices of an edge that
link different blocks of $\pi$ are colored differently. 
Let $G_{\pi}$ the graph obtained from $G$ by merging all vertices that are
contained in one and the same block of $\pi$. Then we see that the function
$f_{G}$ can be represented as a chromatic polynomial by $f_{G}\left(
\pi,x\right)  =P\left(  G_{\pi},x\right)  $.
For each coloring $\phi:V\rightarrow X$ there exists a unique
$\pi\in\Pi_{c}\left(  G\right)  $ such that $\phi$ contributes to
$P(G_\pi,x)  $. The number of all colorings with at most $x$
colors that assign the same color to all vertices of a block of $\pi$ is
$x^{\left\vert \pi\right\vert }$. Let $\pi\in\Pi_{c}\left(  G\right)
$, then it is%
\begin{align*}
x^{\left\vert \pi\right\vert }=\sum_{\substack{\sigma\geq\pi\\\sigma
\in\Pi_{c}\left(  G\right)  }}P(G_\sigma, x).
\end{align*}
and by M\"{o}bius inversion the equation
\begin{align*}
P(G_\pi,x)&=\sum_{\substack{\sigma\geq\pi\\\sigma\in\Pi_c\left(G\right)}}
\mu\left(  \pi,\sigma\right)
x^{\left\vert \sigma\right\vert }
\end{align*}
is obtained. The chromatic polynomial of $G$ corresponds to $f_{G}\left(  \hat{0},x\right)
=P(G,x)$, which gives Theorem~\ref{thm:rota} presented in \cite{rota:OnTheFoundationsOfCombinatorialTheory}.
\begin{thmenv}[Rota]The chromatic polynomial of a graph $G=\left(  V,E\right)  $ satisfies%
\begin{align*}
P\left(  G,x\right)  =\sum_{\sigma\in\Pi_{c}\left(  G\right)  }%
\mu\left(  \hat{0},\sigma\right)  x^{\left\vert \sigma\right\vert }.
\end{align*}
\label{thm:rota}
\end{thmenv}

\section{The partition polynomial}
The notion of connected set partitions of a graph $G=(V,E)$ is utilized to define the partition polynomial $Q(G,x)$
\begin{align}
Q(G,x)&=\sum_{\pi\in\Pi_c(G)} x^{|\pi|}=\sum_{i=0}^n q_{i}(G)x^i,
\end{align}
where the coefficients $q_{i}(G)$ count the number of connected set partitions of $G$ with $i$ blocks, e.g. consider the example graph $G$ depicted in Figure~\ref{fig:exampleGraph}, page~\pageref{fig:exampleGraph} then it is
\begin{align*}
Q(G,x)&=x^6+9x^5+28x^4+35x^3+15x^2+x.
\end{align*}
Define $\emptyset=(\emptyset,\emptyset)$ to be the null graph. It is  $Q(\emptyset, x)=1$, which corresponds to the only connected set partition of the null graph - the empty set partition $\emptyset$, that has no blocks at all. For no graph other than the null graph, the empty set partition is an element of $\Pi(V)$, which gives $q_0(G)=\delta_{0,n}$. For the empty graph (edgeless graph) $E_n=(V, \emptyset)$ with $n$ vertices the partition polynomial $Q(E_n, x)=x^{n}$ is obtained. The only connected set partition in $\Pi_c(E_n)$ is $\hat{0}$. The set partition $\hat{0}$ is a connected set partition for any graph. Hence the partition polynomial has degree $n=|V|$. The greatest set partition $\hat{1}\in\Pi(V)$, consisting of one block, namely $V$ itself, is connected iff $G$ is connected. Thus it is
\begin{align*}
q_1(G)&=\begin{cases}
1 & G\text{ is connected}  \\
0 & \text{else}.
\end{cases}
\end{align*}
More generally, the minimal degree of the partition polynomial (the least appearing power in $Q(G,x)$) equals the number of components of $G$.

In the complete graph $K_n$ with $n$ vertices, each set partition of the vertex set $V$ is a connected set partition, which implies 
\begin{align*}
Q(K_n,x)&=\sum_{i=0}^n S(n,i)x^i.
\end{align*}
A connected set partition of $G=(V,E)$ with exactly $n-1$ blocks consists of one pair (two-element set) and $n-2$ singletons (one-element sets). The pair corresponds to an edge of $G$. Consequently, $q_{n-1}(G)=|E|$ is obtained. A cutset in a graph $G=(V,E)$ is a subset $C\subseteq E$, such that $(V, E\setminus C)$ has more components than $G$, but no proper subset of $C$ has this property. If $G$ is connected then $q_2(G)$ equals the number of cutsets of $G$.

\begin{lemenv}
Let $G=(V,E)$ be a graph with $n$ vertices and $m$ edges. Then the coefficient
$q_{n-2}(G)$ of the partition polynomial of $G$ satisfies%
\begin{align*}
q_{n-2}\left(  G\right)  =\binom{m}{2}-~2\cdot~\text{number of triangles of
}G.
\end{align*}
\end{lemenv}

\begin{proof}
A connected set partition of $G$ with $n-2$ blocks can have two different
types:\newline(1) Set partitions consisting of two vertex pairs and singletons
else.\newline(2) Set partitions with one tri-element block and singletons.\newline
Each set partition of the first type corresponds to a unique selection of two
non-adjacent edges of $G$. A selection of two edges that have one vertex in
common generates a set partition of the second type. However, if the three
vertices of the block induce a triangle in $G$ then there exist two other
selections of two edges generating the same connected block. The subtraction
of twice the number of triangles of $G$ in the given formula takes this fact
into account.
\end{proof}

In a similar way the following equation for a triangle-free graph
$G$ with $m$ edges is proven:
\begin{align*}
q_{n-3}\left(  G\right)  =\binom{m}{3}-~3\cdot~\text{number of four-cycles of
}G
\end{align*}
As a more general conclusion the following statement is obtained.

Let $G$ be a graph with $m$ edges and $Q\left(  G,x\right)  $ its partition
polynomial. If for $j=0,...,k-1$ the relations $q_{n-j}=\binom{m}{j}$ and
$q_{k}<\binom{m}{k}$ are satisfied then the girth of $G$ equals $k+1$. The
number of cycles of length $k+1$ is in this case%
\begin{align*}
\frac{1}{k}\left[  \binom{m}{k}-q_{n-k}\left(  G\right)  \right]  .
\end{align*}
Dowling and Wilson \cite{dowling:WhitneyNumberInequalitiesForGeometricLattices} proved a theorem on Whitney numbers of the second kind in geometric
lattices that translates directly into an inequality for the coefficients of
the partition polynomial.

\begin{thmenv}
[Dowling and Wilson]The coefficients $q_{i}\left(  G\right)  $ of the
partition polynomial satisfy for any connected graph with $n$ vertices and for each
$k\in\left\{  1,...,n\right\}  $ the inequality%
\begin{align*}
\sum_{i=1}^{k}q_{i}\left(  G\right)  \geq\sum_{i=0}^{k-1}q_{n-i}\left(
G\right)  .
\end{align*}
\end{thmenv}

\begin{lemenv}
\label{lemma_bridge}If $e\in E$ is a bridge of $G=\left(  V,E\right)  $ then%
\begin{align*}
Q\left(  G,x\right)  =Q\left(  G/e,x\right)  +Q\left(  G-e,x\right)  .
\end{align*}
\end{lemenv}

\begin{proof}
Let $e=\left\{  u,v\right\}  $ be a bridge of $G$. Each connected set partition of $G$ belongs to one of two classes:\newline(1) Set partitions that contain a block
$X$ with $u\in X$ and $v\in X$.\newline(2) Set partitions for which $u$ and $v$
belong to different blocks.\newline Each set partition of the first class
corresponds to a connected set partition of $G/e$ that is obtained by replacing
$u$ and $v$ by a single vertex. The set partitions of the second class are exactly the set partitions of $G-e$.
\end{proof}

\begin{corenv}
\label{cor_arti}Let $G=\left(  V,E\right)  $ be a graph and $Q\left(
G,x\right)  $ its partition polynomial. If $v\in V$ is a vertex of degree 1,
then%
\begin{align*}
Q\left(  G,x\right)  =\left(  1+x\right)  Q\left(  G-v,x\right)  .
\end{align*}
\end{corenv}
Corollary \ref{cor_arti} can be applied to compute the partition polynomial of a tree $T_n$ with $n$ vertices; the result is
\begin{align*}
Q\left(  T_{n},x\right)  =\left(  1+x\right)  ^{n-1}x.
\end{align*}
The following decomposition formula is the basis in order to derive the
partition polynomial of an arbitrarily given graph. Let $G=(V,E)$ be a graph and $W\subseteq V$ a vertex subset, then $G-W$ denotes the graph obtained from $G$ by removing all vertices of $W$.

\begin{thmenv}
\label{theo_decom}Let $G=\left(  V,E\right)  $ be a graph and
$v\in V$, then
\begin{align*}
Q\left(  G,x\right)  =x\sum_{\substack{\left\{  v\right\}  \subseteq
W\subseteq V \\G\left[  W\right]  \text{ is conn.}}}Q\left(  G-W,x\right)  .
\end{align*}
Here the sum is taken over all vertex induced connected subgraphs that contain $v$.
\end{thmenv}

\begin{proof}
Let $v\in V$ be a given vertex of $G=\left(  V,E\right)  $. For each connected set partition $\pi=\left\{  X_{1},...,X_{k}\right\}  \in\Pi_{c}\left(
G\right)  $ with $v\in X_{1}$ the induced subgraph $G\left[  X_{1}\right]
$ is connected and $\left\{  X_{2},...,X_{k}\right\}  $ is a connected set
partition of $G-X_{1}$. All connected set partitions that include $X_{1}$ are
counted by $Q\left(  G-X_{1},x\right)  $. The blocks of a connected set
partition $\pi=\left\{  X_{1},...,X_{k}\right\}  \in\Pi_{c}\left(
G\right)  $ can always be renumbered in such a way that $v\in X_{1}$ is
valid. Consequently, $x~Q\left(  G-X_{1},x\right)  $ is the ordinary
generating function for the number of connected set partitions of $G$ that have a block $X_{1}$ containing $v$. Choosing the block containing $v$ in every
possible way (such that the induced subgraph is connected) gives the desired result.
\end{proof}

The application of Theorem \ref{theo_decom} to the partition polynomial of a
cycle $C_{n}$ yields%
\begin{align*}
Q\left(  C_{n},x\right)   & =x\sum_{k=1}^{n-1}k~Q\left(  P_{n-k},x\right)
+x\\
& =x^{2}\sum_{k=1}^{n-1}k\left(  1+x\right)  ^{n-k-1}+x\\
& =\left(  1+x\right)  ^{n}-1-\left(  n-1\right)  x.
\end{align*}
The application of Theorem \ref{theo_decom} requires the summation over all
connected subgraphs containing $v$, which results in an exponentially
increasing computational effort. The following recursion includes only the
neighborhood of a vertex. However, in this case, we need the principle of
inclusion and exclusion.

\begin{thmenv}
\label{theo_pie}For each subset $X\subseteq V,$ let $G/X$ be the graph obtained
from $G=\left(  V,E\right)  $ by merging all vertices of $X$ into a single
vertex. Possibly arising parallel edges are replaced by single edges. Then the
following equation is valid for each vertex $v\in V$:%
\begin{align*}
Q\left(  G,x\right)  =xQ\left(  G-v,x\right)  +\sum_{\emptyset\subset
W\subseteq N\left(  v\right)  }\left(  -1\right)  ^{\left\vert W\right\vert
+1}Q\left(  G/{W\cup\left\{  v\right\}},x\right)  .
\end{align*}
\end{thmenv}

\begin{proof}
The first term on the right hand side of the equation counts all connected set
partitions of $G$ that contain the singleton $\left\{  v\right\}  $. All
remaining connected set partitions of $G$ have the property that $v$ and at least one neighbor vertex of $v$ are together in a block. Let $u\in N\left(
v\right)  $, then $Q\left(  G/{\left\{  u,v\right\}  },x\right)  $ is the
ordinary generating function for the number of connected set partitions of $G$
containing a block $B$ with $\left\{  u,v\right\}  \subseteq B$. However, if
analogously the connected set partitions of $G$ that have a block $B^{\prime}$ with $\left\{v,w\right\}  \subseteq B^{\prime}$ are counted by $Q\left(  G/{\left\{  v,w\right\}  },x\right)  $ with $w\in N\left(  v\right)  $, $w\neq u$, then all connected set partitions with a block containing $\left\{  u,v,w\right\} $ are counted twice. Hence it is necessary to subtract $Q\left(  G/{\left\{u,v,w\right\}  },x\right)  $. By induction on the number of vertices in $N\left(  v\right)  $ the inclusion-exclusion representation as
stated in the theorem is obtained.
\end{proof}

The Theorem~\ref{thm:generalCutSetFormula} is a generalization of Theorem~\ref{theo_pie} and also contains Lemma~\ref{lemma_bridge} as a special case.
\begin{thmenv}\label{thm:generalCutSetFormula}
For an edge subset $F\subseteq E$ let $G/S$ the graph obtained from $G=(V,E)$ by contracting all edges of $S$ in $G$. If $S\subseteq E$ is a cut of $G$, then
\begin{align}
Q\left(  G-S,x\right)&=\sum_{F\subseteq S}\left(  -1\right)  ^{\left\vert
F\right\vert }Q\left(  G/S,x\right) \label{cut_formula}%
\end{align}
holds.
\end{thmenv}
\begin{proof}
Let $F$ be a subset of the given cut $S$ and let $q_{i}\left(  G,F\right)  $
be the number of connected set partitions of $G$ with exactly $i$ blocks such that
the end vertices of any edge of $F$ are contained completely in one block. We
define the polynomial%
\begin{align*}
Q_{F}\left(  G,x\right)&=\sum_{i=1}^{n}q_{i}\left(  G,F\right)  x^{i}.
\end{align*}
If $e=\left\{  u,v\right\}  $ is an edge of $G$ then the polynomials
$Q_{\left\{  e\right\}  }\left(  G,x\right)  $ and $Q\left(  G/e,x\right)  $
coincide, since in both cases we count only connected set partitions of $G$ having
$u$ and $v$ in one block. The generalization of this statement yields%
\begin{align*}
Q_{F}\left(  G,x\right)&=Q\left(  G/F,x\right)
\end{align*}
for all $F\subseteq S$. Let $r_{i}\left(  G,F\right)  $ be the number of
connected set partitions $\pi\in\Pi_c\left(  G\right)  $ with $\left\vert
\pi\right\vert =i$ such that the end vertices of any edge of $F$ are contained
completely in one block \emph{and} such that no two end vertices of any edge
of $S\setminus F$ belong to one and the same block of $\pi$. We consider the
generating function for these number sequence, i.e.%
\begin{align*}
R_{F}\left(  G,x\right)&=\sum_{i=1}^{n}r_{i}\left(  G,F\right)  x^{i}.
\end{align*}
For any subset $F\subseteq S$, we conclude%
\begin{align*}
Q\left(  G/F,x\right)&=\sum_{A\supseteq F}R_{A}\left(  G,x\right)  ,
\end{align*}
which yields via M\"{o}bius inversion%
\begin{align*}
R_{F}\left(  G,x\right)&=\sum_{A\supseteq F}\left(  -1\right)  ^{\left\vert
A\right\vert -\left\vert F\right\vert }Q\left(  G/A,x\right)  .
\end{align*}
The polynomial $R_{\emptyset}\left(  G,x\right)  $ counts all connected set
partitions of $G$ that have no edge of $S$ as a subset of a block. Hence each
block of a set partition counted by $R_{\emptyset}\left(  G,x\right)  $ is
completely contained in one component of $G-S$, which gives
\begin{align*}
Q\left(  G-S,x\right)&=R_{\emptyset}\left(  G,x\right)  =\sum_{A\subseteq
S}\left(  -1\right)  ^{\left\vert A\right\vert }Q\left(  G/A,x\right). \qedhere
\end{align*}
\end{proof}
It is possible to state the partition polynomial of the complete bipartite graph $K_{s,t}$ in a closed form by the following theorem.
\begin{thmenv}
The partition polynomial of the complete bipartite graph is%
\begin{align}
Q\left( K_{s,t},x\right) =\sum_{i=0}^{s}\sum_{j=0}^{s-i}\sum_{k=0}^{t-j}%
\dbinom{s}{i}
\dbinom{t}{k}
S(s-i,j)
S(t-k,j)j!
x^{i+j+k}. 
\end{align}
\end{thmenv}

\begin{proof}
Assume the vertex set of the complete bipartite graph $K_{s,t}$ is $S\cup T$
such that $\left\vert S\right\vert =s$ and $\left\vert T\right\vert =t$ and
each edge of $K_{s,t}$ links a vertex of $S$ with a vertex of $T$. First we
select a vertex subset $X$, $X\subseteq S$ of size $i$ and a vertex subset $Y
$, $Y\subseteq T$ of size $k$. There are $\dbinom{s}{i}\dbinom{t}{k}$
possibilities for this selection. These vertices form singletons of the
connected partition. The remaining $s-i$ vertices of $S$ are partitioned
into $j$ blocks. A second partition with $j$ blocks is generated out of $%
T\setminus Y$. These partitions are counted by the Stirling numbers of the
second kind, more precise by the product $S(s-i,j)S(t-k,j)$. We form a bipartite graph with exactly $j$ components and
without isolated vertices with the vertex set $\left( S\setminus X\right)
\cup \left( T\setminus Y\right) $. The vertex set of one component of the
bipartite graph consists of one block of a partition of $S\setminus X$ and
one block of a partition of $T\setminus Y$. These blocks can be assigned to
each other in $j!$ different ways. The number of blocks of the resulting
connected partition of $K_{s,t}$ is $i+j+k$, which is taken into account by
the power of $x$. The triple sum counts all possible distributions of
subsets and partitions.
\end{proof}

\begin{corenv}
The complete bipartite graph $K_{1,t}$ and $K_{2,t}$ satisfy, respectively,%
\begin{eqnarray*}
Q\left( K_{1,t},x\right)  &=&x\left( 1+x\right) ^{t}\text{,} \\
Q\left( K_{2,t},x\right)  &=&x\left[ \left( 1+x\right) ^{t}-x^{t}\right]
+x^{2}\left[ \left( 2+x\right) ^{t}-x^{t}\right] +x^{2+t}.
\end{eqnarray*}
\end{corenv}

In order to state the Theorem~\ref{thm:daggerTheorem} it is necessary to introduce the notion of the extraction $G\dagger e$ of an edge $e\in E$ from a given graph $G=(V,E)$, as done in \cite{averbouchMakowskyTittmann:GraphPolynomial}. 
Assume that $e=\{u,v\}\in E$, then $G\dagger e$ denotes the graph that emerges from $G$ by removing the edge $e$ and the two endvertices $u$ and $v$ from $G$. This extraction of edges is easily generalized to arbitrary matchings $M\subseteq E$ of $G$ by successively extracting all matching edges in $M$ from $G$.
\begin{thmenv}\label{thm:daggerTheorem}
Let $G=(V,E)$ be a graph and $M\subseteq E$ a matching in $G$, so that every matching edge $e=\{u,v\}\in M$ has the property that $u$ and $v$ are connected to every other vertex $w$ in $V$. Then
\begin{align}
Q(G-M,x)&=\sum_{I\subseteq M}(-x)^{|I|}Q(G\dagger I, x)
\end{align}
holds.
\end{thmenv}
\begin{proof}
Let $e=\{u,v\}\in M$ and denote by $\pi_e$ the set of all connected set partitions of $G$ containing the two element block $\{u,v\}$. Then one has by virtue of the inclusion-exclusion principle
\begin{align}
\left|
\bigcap_{e\in M}\overline{\pi}_e
\right|&=
\sum_{I\subseteq M}
(-1)^{|I|}\left|\bigcap_{e\in I}\pi_e\right|,
\end{align}
where the universal set is given by the set of all connected set partitions in $G$.

For the right hand side note that the generating function of the connected set partitions in $\bigcap_{e\in I}\pi_e$ is given by $x^{|I|}Q(G\dagger I, x)$, as the pairwise disjoint two element blocks created by the matching edges $e\in I$ are yielding $x^{|I|}$ extra blocks, whereas $G\dagger I$ accounts for the removal of these two element blocks. 

For the left hand side one has to show that $\pi\in\bigcap_{e\in M}\overline{\pi}_e$ iff $\pi$ is a connected set partition in $G-M$. 

Let $\pi\in\bigcap_{e\in M}\overline{\pi}_e$ and $B$ a block of $\pi$. Firstly, assume that $s$ and $t$ are distinct vertices in $B$ that are not both endpoints of an matching edge in $M$. As $B$ induces a connected subgraph in $G$, there is an $s$-$t$ path $P$ in $G[B]$, that possibly uses edges of $M$. Assume that this path is given by $P=(v_0,e_0,v_1,\ldots, v_{n+1})$ with $v_0=s$ and $v_{n+1}=t$. Let $e_i=\{v_i,v_{i+1}\}$ be the possible first matching edge in $P$, then one can shortcut the $s$-$t$ path to $P=(v_0,\ldots, v_{i}, \{v_i, v_{n+1}\}, v_{n+1})$, as the endpoints of every matching edge are connected to every other vertex in $G$. Furthermore note that the edge $\{v_i, v_{n+1}\}$ cannot be a matching edge, as $e_i$ and $\{v_i, v_{n+1}\}$ would be incident. Hence $s$ and $t$ are also connected in the subgraph induced by $B$ in $G-M$.

Secondly, assume that $s$ and $t$ are distinct vertices in $B$ that are endpoints of an matching edge in $M$. By assumption there must be a third vertex $u$ distinct from $s$ and $t$, so that the non matching edges $\{s,u\}$ and $\{u,t\}$ exist. Hence $P=(s,\{s,u\},u,\{u,t\},t)$ is an $s$-$t$ path in the subgraph induced by $B$ in $G-M$ . 

Therefore the removal of the edges in $M$ will not destroy the connectivity of $B$ and $\pi$ is also a connected set partition in $G-M$.

Conversely suppose that $\pi$ is a connected set partition in $G-M$, then $\pi$ will not contain any block corresponding to matching edges in $M$ and adding the edges in $M$ cannot destroy the connectivity of the blocks of $\pi$, so that $\pi$ is also a connected set partition in $\bigcap_{e\in M}\overline{\pi}_e$.

By considering the generating function $Q(G-M,x)$ of the set of connected set partitions of $G-M$ the result for the left hand side follows.
\end{proof}
An application of Theorem~\ref{thm:daggerTheorem} is given by the  Corollary~\ref{cor:applicationDaggerTheorem}, where by the symmetry of the $K_n$ some simplifications are possible.
\begin{corenv}\label{cor:applicationDaggerTheorem}
Let $M_m$ be a matching with exactly $m$ edges of the complete graph $K_n$. Then
\begin{align}
Q(K_n-M_m,x)&=\sum_{i=0}^m\sum_{j=0}^{n-2i} \binom{m}{i}(-1)^iS(n-2i,j)x^{i+j}.
\end{align}
\end{corenv}

\section{The splitting formula}
The aim of this section is to derive a more general computational device to compute the partition polynomial of an arbitrarily given graph $G$. The splitting formula on a separating vertex set $X$ presented in this section can be employed to construct an algorithm with polynomial running time that computes the partition polynomial for graphs with bounded treewidth.

Given a graph $G=(V,E)$ and subgraphs $G^1=(V^1,E^1)$, $G^2=(V^2,E^2)$ of G, so that $E=E^1\cup E^2$, $E^1\cap E^2=\emptyset$ and $V=V^1\cup V^2$, $V^1\cap V^2=X$ is satisfied, then $s(G)=(G^1,G^2,X)$ is a splitting of $G$. The set $X$ is a separating vertex set of $G$. 

Let $X\subseteq V$, then $\Pi_c(G, \beta)$ for $\beta\in\Pi(X)$ denotes the set of connected set partitions of $G$, that are inducing $\beta$ in $X$. Let $G_X$ be the graph emerging from $G$ by adding edges between all vertices of $X$ in $G$, so that $G_X[X]$ becomes the complete graph $K_X$ on the vertex set $X$.

If $\pi=\{V_1, \ldots, V_k\}\in\Pi(V)$, then $\pi$ induces the connected set partition 
\begin{align*}
\{G[\pi]\}=\bigcup_{j=1}^k \{G[V_j]\}
\end{align*}
in $G$ by the connected components of the subgraphs $G[V_j]$ induced by the blocks $V_j$. It is $\{G[\pi]\}=\pi$ iff $\pi$ is a connected set partition of $G$. Note that this definition also implies the inequality $\{G[\pi]\}\le\pi$ for all $\pi\in\Pi(V)$. Consider the set partition $\pi=\{\{a,b,e,f\},\{c,d\}\}$ and the graph depicted in Figure~\ref{fig:exampleGraph} on page~\pageref{fig:exampleGraph} then $\{G[\pi]\}=\{\{a,b\},\{e,f\},\{c,d\}\}$ as shown in Figure~\ref{fig:inducedPartitionGraph}.
\begin{figure}[h]
\begin{center}
\includegraphics[scale=0.6]{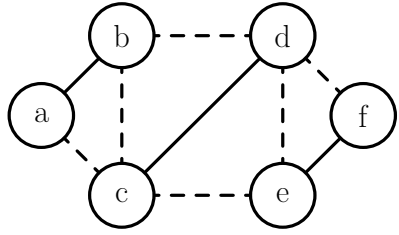}
\caption{The connected set partition $\{G[\pi]\}$ induced by $\pi$ \label{fig:inducedPartitionGraph}}
\end{center}
\end{figure}
Note that the block $\{a,b,e,f\}\in\pi$ splits into the blocks $\{a,b\}$ and $\{e,f\}$, as the former block does not induce a connected subgraph in the underlying graph $G$.

Let $G=(V,E)$ be a graph, $X\subseteq V$ and $\gamma,\beta\in\Pi(X)$, then 
$T(G,\beta,\gamma)$ is defined as the set
\begin{align}\label{def:Tset}
T(G,\beta,\gamma)&=\{\pi\in\Pi_c(G_X,\beta): \{G[\pi]\}\sqcap X=\gamma\}
\end{align} 
and the associated counting polynomial $T(G,\beta,\gamma;x)$ as
\begin{align}\label{def:TPolynomials}
T(G,\beta,\gamma;x)&=\sum_{\pi\in T(G,\beta,\gamma)}x^{|\pi|}.
\end{align}

Observe that if not $\gamma\le\beta$ holds one has $T(G,\beta,\gamma)=\emptyset$ and therefore also $T(G,\beta,\gamma;x)=0$. Furthermore note that in the case of $X=\emptyset$ it is 
\begin{align*}
Q(G,x)&=T(G,\emptyset,\emptyset;x).
\end{align*}

It is notable that the connected set partitions of a graph $G=(V,E)$ can be partitioned into disjoint sets by considering the set partitions $\beta$ that they are inducing in $X\subseteq V$. Therefore the $T(G,\beta,\beta;x)$ polynomials can be used to determine the partition polynomial by the sum
\begin{align}\label{eqn:calculatingTheQPolynomialsByAddingRestrictedTPolynomials}
Q(G,x)&=\sum_{\beta\in\Pi(X)} T(G,\beta,\beta;x).
\end{align}
The central Lemma~\ref{lemma:GlueingLemma} shows how a connected set partition in $G=(V,E)$ is assembled, given a splitting $s(G)=(G^1,G^2,X)$ of $G$ and the connected set partitions in the subgraphs $G^1$ and $G^2$.

\begin{lemenv}\label{lemma:GlueingLemma}
Let $G=(V,E)$ be a graph and $s(G)=(G^1,G^2,X)$ a splitting of $G$. 

Let $\sigma\in\Pi_c(G_X,\beta)$ and $\sigma^1=\sigma \sqcap V^1$, $\sigma^2=\sigma\sqcap V^2$, respectively. Then $\sigma^1\in\Pi_c(G_X^1,\beta)$ and $\sigma^2\in\Pi_c(G_X^2,\beta)$ holds, satisfying $\sigma=\sigma^1\sqcup\sigma^2$ and $\{G[\sigma]\}=\{G^1[\sigma^1]\}\sqcup \{G^2[\sigma^2]\}$.

Conversely let $\sigma^1\in\Pi_c(G^1_X,\beta)$ and $\sigma^2\in\Pi_c(G^2_X,\beta)$, then $\sigma=\sigma^1\sqcup\sigma^2\in\Pi_c(G_X,\beta)$ satisfying $\{G[\sigma]\}=\{G^1[\sigma^1]\}\sqcup \{G^2[\sigma^2]\}$.
\end{lemenv}
\begin{proof}
$\Rightarrow:$
Let $\sigma\in\Pi_c(G_X,\beta)$ with $\sigma^1=\sigma\sqcap V^1$ and $\sigma^2=\sigma\sqcap V^2$. Consider a block $S_j\in\sigma$, then the construction of $S_j$ can be distinguished in the case (a) $X\cap S_j=\emptyset$  and the case (b) $X\cap S_j\ne\emptyset$, as been exemplified in Figure~\ref{fig:faelle}.

Note that $G^1_X[S_j\cap V^1]$, $G^2_X[S_j\cap V^2]$ are connected, possibly empty, subgraphs of $G^1_X$, $G^2_X$, for the given block $S_j$, by considering the distinct cases $(a)$ and $(b)$. Even more observe that $\beta=\sigma^1\sqcap X=\sigma^2\sqcap X$ holds, so that $\sigma^1\in\Pi_c(G_X^1,\beta)$ and $\sigma^2\in\Pi_c(G_X^2,\beta)$ is satisfied.

In the case (a) it is $S_j\cap X=\emptyset$. Set $S_j^1=S_j\cap V^1\in\sigma^1$ and $S_j^2=S_j\cap V^2\in\sigma^2$. As $X$ is a separating vertex set of $G$ either $S_j^1=S_j$ and $S_j^2=\emptyset$ or $S_j^2=S_j$ and $S_j^1=\emptyset$ can occur, as being exemplified in Figure~\ref{fig:faelle} by the blocks $\{s,t,u,v\}\in\sigma^1$ and $\{r,o,p,q\}\in\sigma^2$. Hence the conditions $\{S_j\}=\{S_j^1\}\sqcup \{S_j^2\}$ and 
$\{G[S_j]\}=\{G^1[S_j^1]\}\sqcup \{G^2[S_j^2]\}$ are satisfied.

In the case (b) it is $S_j\cap X\ne\emptyset$, so that $S_j^1=S_j\cap V^1\in\sigma^1$, $S_j^2=S_j\cap V^2\in\sigma^2$ and $S_j=S_j^1\cup S_j^2$ holds with $S_j^1\cap S_j^2=B_j\in\beta$, as $X$ is a separating vertex set of $G$ and as $\sigma$ induces $\beta$ in $X$. Note that $\{S_j\}=\{S_j^1\}\sqcup \{S_j^2\}$ is satisfied.
An example for this case is found in Figure~\ref{fig:faelle} through the block $S_j=\{a,b,i,j,k,h\}$, which is the union of the blocks $S_j^1=\{a,b,i,h\}$ and $S_j^2=\{a,b,j,k\}$. Even more the relation $\{G[S_j]\}=\{G^1[S_j^1]\}\sqcup\{G^2[S_j^2]\}$ holds. This case is exemplified in Figure~\ref{fig:fallb} through $\{G[S_j]\}=\{\{a,e,g,h,i\},\{b,c,d,f,j\}\}$, $\{G^1[S_j^1]\}=\{\{a,e\},\{b,c,f\},\{d\}\}$ and, $\{G^2[S_j^2]\}=\{\{a,g,h,i\},\{b\},\{c,d,j\}\}$.

For all $S_j\in\sigma$ the conditions $\{S_j\}=\{S_j^1\}\sqcup \{S_j^2\}$ and
$\{G[S_j]\}=\{G^1[S_j^1]\}\sqcup \{G^2[S_j^2]\}$ hold with $S_j^1=S_j\cap V^1\in\sigma^1$ and $S_j^2=S_j\cap V^2\in\sigma^2$. Recalling that a block $S_j\in\sigma$ can either decompose into one block (a) in $\sigma^1$ or $\sigma^2$ or two blocks (b) in $\sigma^1$ and $\sigma^2$, as $X$ is a separating vertex set, the desired claim with $\sigma=\sigma^1\sqcup\sigma^2$ and $\{G[\sigma]\}=\{G^1[\sigma^1]\}\sqcup\{G^2[\sigma^2]\}$ follows.

$\Leftarrow:$ Note that the set partition $\sigma=\sigma^1\sqcup \sigma^2$ is a connected set partition in $G_X$ with $\sigma\sqcap X=\beta$, by considering the cases (a) and (b), depicted in Figure~\ref{fig:faelle}.

In the case (a) it is assumed without loss of generality that $S_j^1\in\sigma^1$ holds with $S_j^1\cap X=\emptyset$ and $S_j^2=\emptyset$. It is then $\{S_j\}=\{S_j^1\}=\{S_j^1\}\sqcup\{S_j^2\}$ and even more $\{G[S_j]\}=\{G^1[S_j^1]\}\sqcup\{G^2[S_j^2]\}$, where $S_j=S_j^1\cup S_j^2\in\sigma$ is set.

In the case (b) let $S_j^1\in\sigma^1$ and $S_j^2\in\sigma^2$ be the blocks, so that $S_j^1\cap S_j^2=B_j\in\beta$ holds and set $S_j=S_j^1\cup S_j^2\in\sigma$, then $\{S_j\}=\{S_j^1\}\sqcup \{S_j^2\}$ is satisfied and even more $\{G^1[S_j^1]\}\sqcup \{G^2[S_j^2]\}=\{G[S_j]\}$ is true as being exemplified in Figure~\ref{fig:fallb}.

The application of the above argumentation to all blocks $S_j$ of $\sigma$ yields the proof of the converse claim.
\end{proof}
\begin{figure}[h]
\begin{center}
\includegraphics[scale=0.40]{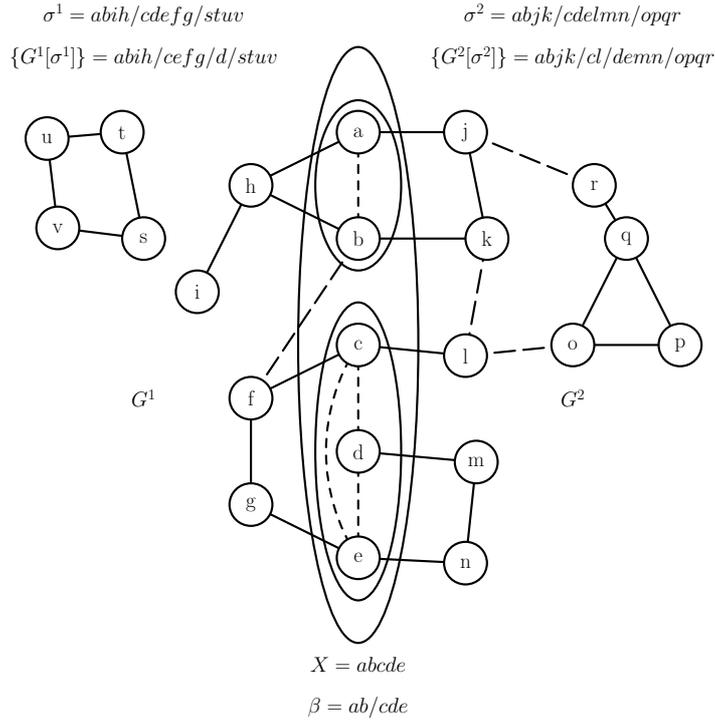}
\caption{Exemplification of the two different cases for the blocks\label{fig:faelle}}
\end{center}
\end{figure}

\begin{figure}[h]
\begin{center}
\includegraphics[scale=0.40]{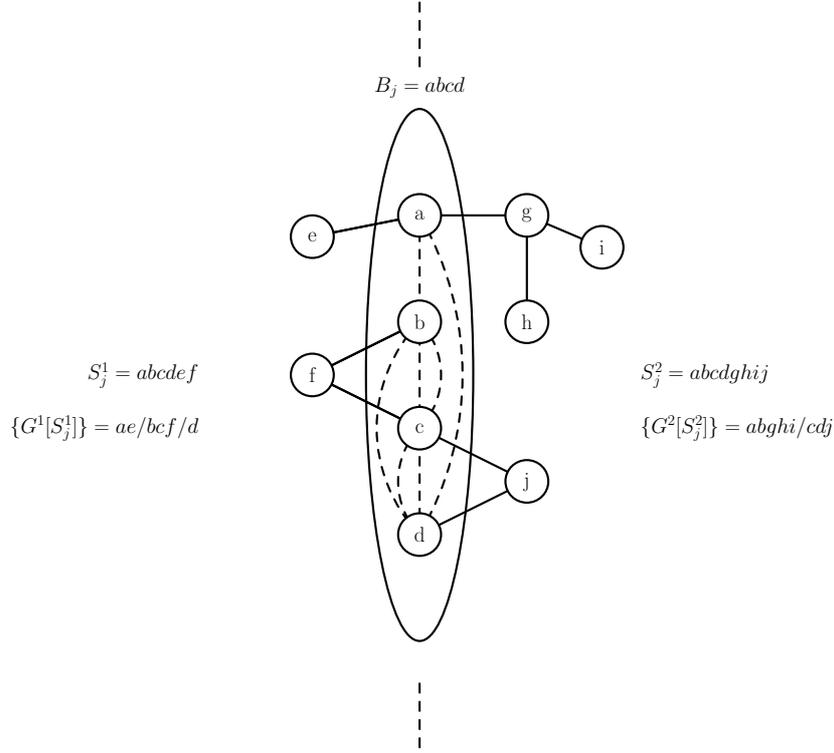}
\caption{Case (b) of the proof\label{fig:fallb}}
\end{center}
\end{figure}

Lemma~\ref{lemma:GlueingLemma} is then utilized to deduce  Theorem~\ref{thm:splittingFormel}, that allows the computation of $Q(G,x)$ by considering the subgraphs $G^1$ and $G^2$ of a splitting $s(G)=(G^1,G^2,X)$.
\begin{thmenv}\label{thm:splittingFormel}
Let $G=(V,E)$ be a graph, $s(G)=(G^1, G^2,X)$ a splitting of $G$ and $(\gamma,\beta)\in\Pi^2(X)$. Then
\begin{align}\label{eqn:splittingFormel}
T(G,\beta,\gamma;x)&=\frac{1}{x^{|\beta|}}\sum_{\gamma^1\vee\gamma^2=\gamma} 
T(G^1,\beta,\gamma^1;x)T(G^2,\beta,\gamma^2;x).
\end{align}
\end{thmenv}
\begin{proof}
Due to Lemma~\ref{lemma:GlueingLemma} every two connected set partitions $\sigma^1$, $\sigma^2$ in the subgraphs $G_X^1$, $G_X^2$ inducing $\beta$ in $X$,  constitute a connected set partition $\sigma=\sigma^1\sqcup \sigma^2$ in $G_X$ that induces $\beta$ in $X$ satisfying $\{G[\sigma]\}=\{G^1[\sigma_1]\}\sqcup\{G^2[\sigma_2]\}$, and even more every connected set partition of $G_X$ inducing $\beta$ in $X$ emerges in this way. 

Setting $\gamma^1=\{G^1[\sigma^1]\}\sqcap X$, $\gamma^2=\{G^2[\sigma^2]\}\sqcap X$ and $\gamma=\{G[\sigma]\}\sqcap X$ yields under the view of $\{G[\sigma]\}=\{G^1[\sigma_1]\}\sqcup\{G^2[\sigma_2]\}$ the condition $\gamma=\gamma^1\vee\gamma^2$.

Therefore the product $T(G^1, \beta,\gamma^1,x)T(G^2, \beta,\gamma^2,x)$ counts the number of connected set partitions $\sigma$ in $G_X$ that are inducing $\beta$ in $X$ satisfying $\gamma^1=\{G^1[\sigma^1]\}\sqcap X$, $\gamma^2=\{G^2[\sigma^2]\}\sqcap X$ and, $\gamma=\gamma^1\vee\gamma^2=\{G[\sigma]\}\sqcap X$. However there is an overcounting done by simply taking the product, as every block of the connected set partition of $G_X$ that has a nonempty intersection with $X$ is counted twice by both factors of the product, which complies with the case (b) in the proof of Lemma~\ref{lemma:GlueingLemma}. Therefore it is necessary to divide the product by the factor $x^{|\beta|}$, as there are exactly $|\beta|$ blocks with this property. The summation ranges over all $\gamma^1$ and $\gamma^2$ with $\gamma^1\vee\gamma^2=\gamma$, as only the connected set partitions $\sigma$ with $\{G[\sigma]\}\sqcap X=\gamma$ are counted by $T(G,\beta,\gamma,x)$.
\end{proof}

Finally the announced splitting formula is deduced by application of Theorem~\ref{thm:splittingFormel} and the Equation~\eqref{eqn:calculatingTheQPolynomialsByAddingRestrictedTPolynomials}.
\begin{thmenv}\label{thm:splittingFormula}
Let $G=(V,E)$ be a graph and $s(G)=(G^1,G^2,X)$ a splitting of $G$. Then
\begin{align}
Q(G,x)&=\sum_{\beta\in\Pi(X)}
\frac{1}{x^{|\beta|}}
\sum_{\gamma^1\vee\gamma^2=\beta} 
T(G^1,\beta,\gamma^1;x)T(G^2,\beta,\gamma^2;x).
\end{align}
\end{thmenv}
\begin{proof}
Just substitute Equation~\eqref{eqn:splittingFormel} in Equation~\eqref{eqn:calculatingTheQPolynomialsByAddingRestrictedTPolynomials} to get the desired result.
\end{proof}
\noindent The computational requirements of Theorem~\ref{thm:splittingFormula} can be reduced, if the intersection of the subgraphs $G^1$ and $G^2$ is the complete graph $K_X$ of the separating vertex set $X$. In this case the polynomials $T(G^1,\beta,\gamma;x)$ and $T(G^2,\beta, \gamma;x)$ will vanish for all $\gamma$ with $\gamma\ne\beta$.
\begin{corenv}\label{Corollary:SpecialCaseCompleteGraph}
Let $G=(V,E)$ be a graph and $s(G)=(G^1,G^2,X)$ a splitting, so that $G[X]\simeq K_{|X|}$ holds. Then
\begin{align}\label{eqn:splittingFormulaCompleteGraph}
Q(G,x)&=\sum_{\beta\in\Pi(X)}\frac{1}{x^{|\beta|}}
T(G^1,\beta,\beta;x)T(G^2,\beta,\beta;x).
\end{align}
\end{corenv}
\noindent Corollary~\ref{Corollary:SpecialCaseCompleteGraph} covers two special cases, that exhibit a special simple form, namely $X=\emptyset$ and $X=\{v\}$. As in the case of two disjoint graphs $G^1$ and $G^2$ it is  
\begin{align*}
Q(G,x)&=Q(G^1,x)Q(G^2,x)
\end{align*}
and in the case of an articulation $v$
\begin{align*}
Q(G,x)&=\frac{1}{x}Q(G^1,x)Q(G^2,x).
\end{align*}
Observe that the above two formulas also occur in the computation of the chromatic polynomial $P(G,x)$ of the graph $G$, e.g.
\begin{align}
P(G,x)&=P(G^1,x)P(G^2,x)
\end{align}
and
\begin{align}
P(G,x)&=\frac{1}{x}P(G^1,x)P(G^2,x).
\end{align}
are valid as well.

\section{Clusterings in graphs}
This brief section presents a possible generalization of the partition polynomial, the minimal cut polynomial. Most formulas derived in the previous sections are readily transfered to this more general polynomial.

Let $G=(V,E)$ be a graph and $\pi=\{V_1, \ldots, V_k\}\in\Pi_c(G)$ a connected set partition. Denote by $E(G[V_i])$ the edge set of the subgraph $G[V_i]$ induced by $V_i$ for all $1\le i\le k$. Then 
\begin{align*}
E(G[\pi])&=\bigcup_{i=1}^k E(G[V_i])
\end{align*}
denotes the set of intra block edges, where $\pi$ is a connected set partition in $G$.

Let $G=(V,E)$ be a graph with $n$ vertices and $m$ edges. The bivariate minimal cut polynomial $Q(G,x,y)$ is an extension of the partition polynomial that accounts for the edges by introducing the counting variable $y$:
\begin{align}
Q(G,x,y)&=\sum_{\pi\in\Pi_c(G)}x^{|\pi|}y^{|E(G[\pi])|}\\
&=\sum_{i=0}^n\sum_{j=0}^m q_{ij}(G)x^iy^j
\end{align}
The coefficients $q_{ij}(G)$ can be interpreted as the number of connected set partitions of $V$ with $i$ blocks and $j$ edges, whose endvertices belong to the same block or as the number of clusterings in $G$ with $i$ blocks and $j$ intra-cluster edges. Note that $Q(G,x)=Q(G,x,1)$, so that the minimal cut polynomial is an extension of the partition polynomial. 

The minimal cut polynomial of the graph $G=(V,E)$ depicted in Figure~\ref{fig:exampleGraph} is 
\begin{align*}
Q(G,x,y)&=x^6+9yx^5+(4y^3+24y^2)x^4+(3y^5+12y^4+20y^3)x^3+\\
&\,(2y^7+5y^6+4y^5+4y^4)x^2+y^9x.
\end{align*}

In analogy to Corollary~\ref{cor_arti}, vertices with degree one can be reduced.
\begin{corenv}\label{cor:vertexOfDegreeOneMinimalCutPolynomial}
Let $G=(V,E)$ be a graph and $Q(G,x,y)$ the minimal cut polynomial of $G$. If $v\in V$ is a vertex of degree 1, then
\begin{align*}
Q(G,x,y)&=(y+x)Q(G-v,x,y)
\end{align*}
\end{corenv}
\noindent Let $T_n$ be an arbitrary tree with $n\ge 1$ vertices, then Corollary~\ref{cor:vertexOfDegreeOneMinimalCutPolynomial} can be applied $n-1$ times, hence
\begin{align*}
Q(T_n,x,y)&=(x+y)^{n-1}x.
\end{align*}
In the same vein Theorem~\ref{theo_decom} generalizes to the minimal cut polynomial.
\begin{thmenv}\label{thm:removalOfOneVertex}
Let $G=(V,E)$ be a graph and $v\in V$, then
\begin{align*}
Q(G,x,y)&=x\sum_{\substack{\{v\}\subseteq W\subseteq V \\ G[W] \text{ is connected} }}
y^{|E(G[W])|}Q(G-W,x,y).
\end{align*}
\end{thmenv}
\noindent The application of Theorem~\ref{thm:removalOfOneVertex} to the cycle $C_n$ of length $n$ yields
\begin{align*}
Q(C_n,x,y)&=xy^{n}+x\sum_{k=1}^{n-1}ky^{k-1}Q(P_{n-k};x,y) \\
&=xy^{n}+x^2\sum_{k=1}^{n-1}ky^{k-1}(x+y)^{n-k-1}\\
&=(x+y)^n-xy^{n-1}(n-y)-y^n.
\end{align*}
Similarly Theorem~\ref{thm:removalOfOneVertex} can be applied to the complete graph $K_n$ on $n$ vertices, yielding the recurrence relation
\begin{align*}
Q(K_n,x,y)&=x\sum_{k=1}^n \binom{n-1}{k-1} y^{\binom{k}{2}} Q(K_{n-k},x,y)
\end{align*}
with the initial condition $Q(K_0;x,y)=1$. 

It is notable that the exponential generating function of the sequence $\{Q(K_n,x,y)\}$
\begin{align*}
\hat{H}(z;x,y)&=\sum_{n\ge 0} Q(K_n;x,y)\frac{z^n}{n!}
\end{align*}
is conveyed by application of the exponential formula~\cite{wilf:generatingfunctionology}
\begin{align*}
\hat{H}(z;x,y)&=\exp\left(\sum_{i\ge 1} xy^{\binom{i}{2}}\frac{z^i}{i!}
\right),
\end{align*}
so that the sum representation 
\begin{align*}
Q(K_n,x,y)&=
\sum_{(1^{k_1}\ldots n^{k_n})\vdash n}
\frac{n!}{k_1!\cdots k_n!}
\left(\frac{xy^{\binom{1}{2}}}{1!}\right)^{k_1}
\cdots
\left(\frac{xy^{\binom{n}{2}}}{n!}\right)^{k_n}
\end{align*}
is found, where the sum ranges over all partitions of the number $n$, where the number of parts of size $i$ is given by $k_i$, hence $\sum_{i=1}^n ik_i=n$ holds.

Theorem~\ref{thm:splittingFormula} is also easily transfered to the minimal cut polynomial, where the $T(G,\beta,\gamma,x,y)$ are defined in analogy to the $T(G,\beta,\gamma,x)=T(G,\beta,\gamma,x,1)$. Note that by the definition of the splitting no overcounting of edges occurs that has to be considered, as $E=E^1\cup E^2$ with $E^1\cap E^2=\emptyset$ holds, for every splitting $s(G)=((V^1,E^1),(V^2,E^2),X)$ of $G$.
\begin{thmenv}\label{thm:splittingFormulaGeneralizedPartitionPolynomial}
Let $G=(V,E)$ be a graph and $s(G)=(G^1, G^2, X)$ a splitting of $G$. Then
\begin{align*}
Q(G,x,y)&=\sum_{\beta\in\Pi(X)}\frac{1}{x^{|\beta|}} \sum_{\gamma^1\vee\gamma^2=\beta}
T(G^1,\beta,\gamma,x,y)T(G^2,\beta,\gamma,x,y).
\end{align*}
\end{thmenv}
\noindent In the special cases $X=\emptyset$ and $X=\{v\}$ again
\begin{align*}
Q(G,x,y)&=Q(G^1,x,y)Q(G^2,x,y)
\end{align*}
and 
\begin{align*}
Q(G,x,y)&=\frac{1}{x}Q(G^2,x,y)Q(G^2,x,y)
\end{align*}
is satisfied.

\section{Distinctive power}
Two graphs $G$ and $H$ with coinciding partition polynomials are said to be partition equivalent and likewise chromatically equivalent, if their chromatic polynomials coincide. In the view of Theorem~\ref{thm:rota} it is remarkable that there are graphs that are partition equivalent, but not chromatically equivalent and conversely that there are graphs, which are chromatically equivalent, but not partition equivalent.

\begin{figure}[ht]
\begin{center}
\includegraphics[scale=0.55]{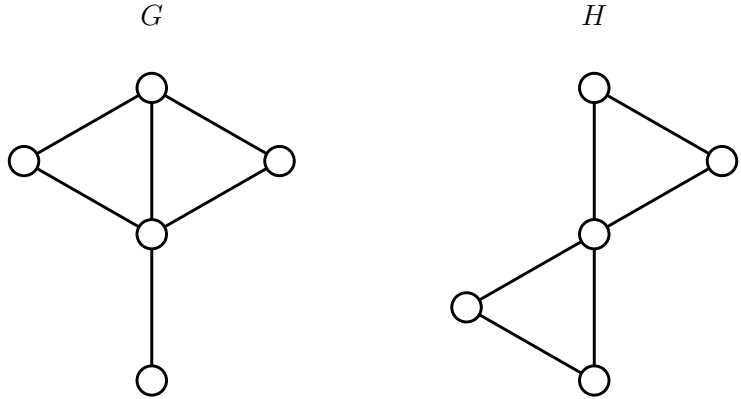}
\caption{Two chromatically equivalent graphs, that are not partition equivalent\label{fig:nonQEquivalentButChromaticallyEquivalent}.}
\end{center}
\end{figure}
\noindent Consider for example the graphs $G$ and $H$ depicted in Figure~\ref{fig:nonQEquivalentButChromaticallyEquivalent} that are chromatically equivalent 
\begin{align*}
P(G,x)=P(H,x)&=x^5-7x^4+18x^3-20x^2+8x,
\end{align*}
but with the distinct partition polynomials
\begin{align*}
Q(G,x)&=x^5+7x^4+15x^3+11x^2+x, \\
Q(H,x)&=x^5+7x^4+15x^3+10x^2+x
\end{align*}
\begin{figure}[ht]
\begin{center}
\includegraphics[scale=0.55]{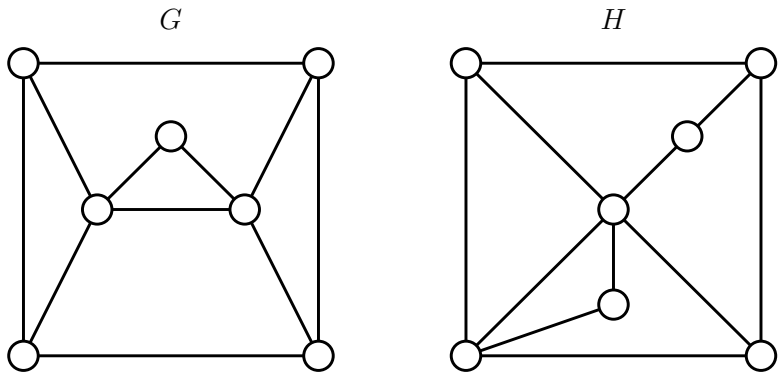}
\caption{Two partition equivalent graphs, that are not chromatically equivalent\label{fig:nonChromEquivalentButQEquivalent}.}
\end{center}
\end{figure}
\noindent and observe that the graphs $G$ and $H$ depicted in Figure~\ref{fig:nonChromEquivalentButQEquivalent} are partition equivalent 
\begin{align*}
Q(G,x)&=Q(H,x)=x^7+11x^6+49x^5+105x^4+100x^3+32x^2+x,
\end{align*}
with the distinct chromatic polynomials
\begin{align*}
P(G,x)&=x^7-11x^6+52x^5-134x^4+196x^3-152x^2+48x, \\
P(H,x)&=x^7-11x^6+52x^5-135x^4+201x^3-160x^2+52x.
\end{align*}
The efficient construction of graphs that are partition equivalent, but not chromatically equivalent remains an open problem. The smallest counterexample depicted in Figure~\ref{fig:nonChromEquivalentButQEquivalent} was found after an exhaustive computer search using a library of nonisomorphic graphs found in~\cite{brendan:listOfGraphs}. 

\section{Complexity}

\begin{thmenv}
The computation of the two-variable polynomial $Q\left( G;x,y\right) $ is
\#P-hard.
\end{thmenv}

\begin{proof}
\textsc{Goldschmidt} and \textsc{Hochbaum} \cite{GH1994} show that the $k$%
-cut problem is \textbf{NP}-complete by reducing it to the maximum clique
problem. The $k$-cut problem defined as follows: Given an undirected graph
and a number $k\in \mathbb{Z}^{+}$, find a minimum edge subset (a $k$-cut)
that, when deleted, separates the graph into exactly $k$ components. Assume
we know the partition polynomial $Q\left( G;x,y\right) $. Then the
coefficient $\left[ x^{k}y^{l}\right] Q\left( G;x,y\right) $ is the number
of edge subsets $A$ such the graph $\left( V,E\setminus A\right) $ has
exactly $k$ components. Consequently, the smallest power of $y$ in $\left[
x^{k}\right] Q\left( G;x,y\right) $ yields the minimum number of edges that
form a $k$-cut of $G$.
\end{proof}

\bibliographystyle{plain}
\nocite{stanley:ASymmetricFunctionGeneralizationOfTheChromaticPolynomialOfAGraph}
\nocite{averbouchMakowskyTittmann:GraphPolynomial}

\end{document}